\documentclass[a4paper,10pt,reqno]{amsart}

\usepackage{etex}
\usepackage[T1]{fontenc}
\usepackage[utf8]{inputenc}
\usepackage{lmodern}
\usepackage{verbatim}
\usepackage[english]{babel}
\usepackage{geometry}
\usepackage[pdftex]{graphicx}
\usepackage{microtype}
\usepackage{float}
\graphicspath{{Images/}}

\usepackage{a4wide}

\usepackage{amsmath,amsfonts,amssymb,amsthm}
\usepackage{mathrsfs,mathtools,accents}
\usepackage{booktabs}
\usepackage{paralist}
\usepackage{subfig}
\usepackage{array}
\usepackage{xy,xcolor,xspace}
\usepackage{multicol}

\usepackage{fancyhdr}
\usepackage{wrapfig}
\usepackage[T1,OT1]{fontenc} 
\usepackage{bm}

\usepackage{grffile}
\usepackage{pgf,tikz}
\usetikzlibrary{matrix}
\usetikzlibrary{shapes.geometric,calc,arrows}
\usepackage[colorlinks,linkcolor=red,citecolor=red]{hyperref}

\usepackage{aliascnt}

\makeatletter
\def\newaliasedtheorem#1[#2]#3{
  \newaliascnt{#1@alt}{#2}
  \newtheorem{#1}[#1@alt]{#3}
  \expandafter\newcommand\csname #1@altname\endcsname{#3}
}
\makeatother

\theoremstyle{plain}
\newaliasedtheorem{thm}[thm1]{Theorem}
\newaliasedtheorem{prop}[thm1]{Proposition}
\newaliasedtheorem{lemma}[thm1]{Lemma}
\newaliasedtheorem{cor}[thm1]{Corollary}
\newaliasedtheorem{counterexample}[thm1]{Counterexample}

\theoremstyle{definition}
\newaliasedtheorem{defn}[thm1]{Definition}
\newaliasedtheorem{question}[thm1]{Question}
\newaliasedtheorem{example}[thm1]{Example}

\theoremstyle{remark}
\newaliasedtheorem{rem}[thm1]{Remark}

\theoremstyle{plain}

\newaliasedtheorem{mainthm}[mainthmDummy]{Theorem}
\newaliasedtheorem{mainprop}[mainthmDummy]{Proposition}

\theoremstyle{definition}
\newaliasedtheorem{maindefn}[mainthmDummy]{Definition}

\newcommand{\R}{\ensuremath{\mathbb{R}}}

\newcommand{\N}{\ensuremath{\mathbb{N}}}

\newcommand{\scal}[2]{\langle #1,#2 \rangle}
\newcommand{\norm}[1]{\left\lVert#1\right\rVert}
\def\Chi#1{\hbox{{\large $\chi$}{\Large $_{_{#1}}$}}}
\newcommand{\st}{\, :\, }

\newcommand{\dx}{\text{d}x}
\newcommand{\dy}{\text{d}y}

\newcommand{\diff}{\text{d}}

\DeclareMathOperator{\Int}{Int}

\newcommand{\grad}{\nabla}
\newcommand{\perim}{\mathcal{P}}
\newcommand{\symmdiff}{\Delta}
\newcommand{\bdry}{\partial}

\newcommand{\lebesgue}{\ensuremath{\mathcal{L}}}
\newcommand{\haus}{\ensuremath{\mathcal{H}}}

\DeclareMathOperator{\supp}{supp}
\DeclareMathOperator{\diam}{diam}

\DeclareMathOperator{\Span}{span}

\DeclareMathOperator*{\argmin}{argmin}

\newcommand{\FF}{\mathcal{F}}	
\newcommand{\GG}{\mathcal{G}}	
\newcommand{\RR}{\mathcal{R}}	
\newcommand{\I}{\mathcal{I}}

\newcommand{\K}{\mathcal{K}}

\def\HH{$(\hbox{\bf H})$}
\def\HI{$(\hbox{\bf I})$}

\begin{document}
\title[A note on some non-local variational problems]{A note on some non-local variational problems}
\author{Davide Carazzato}
\address{Scuola Normale Superiore, Piazza dei Cavalieri 7, 56126 Pisa, Italy}
\email{\href{mailto:davide.carazzato@sns.it}{davide.carazzato@sns.it}.}
\date{}
\begin{abstract}
    We study two non-local variational problems that are characterized by the presence of a Riesz-like repulsive term that competes with an attractive term. The first functional is defined on the subsets of $\R^N$ and has the fractional perimeter $\perim_s$ as attractive term. The second functional instead is defined on  $L^1(\R^N;[0,1])$ and contains an attractive term of positive-power-type. For both of the functionals we prove that balls are the unique minimizers in the appropriate volume constraint range, generalizing the results already present in the literature for more specific energies.
\end{abstract}
\maketitle

\section*{Introduction}
	In this note we deal with two constrained minimization problems, each of them written as sum of an attractive term and a repulsive term. The functionals that we want to minimize are $\FF_{\gamma}:\mathcal{B}(\R^N)\to[0,+\infty]$ and $\GG:\mathcal{K}\to[0,+\infty]$, and are defined as
	\begin{align*}
		\FF_{\gamma}(E) &\coloneqq \perim_s(E) + \gamma\int_E\int_E g(x-y)\dx\dy &\forall E\in\mathcal{B}(\R^N),\\
		\GG(h) &\coloneqq \int_{\R^N}\int_{\R^N}\left(|x-y|^{\alpha}+g(x-y)\right)h(x)h(y)\dx\dy &\forall h\in\mathcal{K},
	\end{align*}
    where $\mathcal{K}\coloneqq L^1(\R^N;[0,1])$ and $\mathcal{B}(\R^N)$ is the Borel $\sigma$-algebra on $\R^N$. In those definitions $\perim_s(E)$ denotes the fractional $s$-perimeter for some fixed $s\in(0,1)$, $\alpha>0$ is a constant, $\gamma>0$ is the coupling parameter between the attractive and the repulsive term and $g:\R^N\setminus\{0\}\to[0,+\infty)$ is an interaction kernel whose properties will be specified later.\\
	Of course, the minimization of those functionals is trivial if we do not put any constraint, so we rewrite the minimization problems respectively as
	\begin{gather}
		\min\left\{\FF_{\gamma}(E)\st E\in\mathcal{B}(\R^N), |E|=\omega_N\right\},\label{eq:FF}\tag{$F$}\\
		\min\left\{\GG(h)\st h\in \mathcal{K}, \norm{h}_1=m\right\}\label{eq:GG}\tag{$G$},
	\end{gather}
	where $\gamma>0$ and $m>0$ are two parameters. Our final aim is to study the behaviour of the minimizers of those problems when $\gamma \ll 1$ and $m\gg 1$. One notices that, even though the parameters $m$ and $\gamma$ play a similar role in those problems, they are technically different if $g$ is non-homogeneous (that is the case we are interested in). In the first problem we vary the coupling parameter, and not the total mass, because our arguments do not apply well to this second case. This is a technical issue, and probably it can be solved imposing some mild growth condition on $g$. Of course, these two approaches are completely equivalent if $g$ is homogeneous, since the volume constraint can be fixed to be $\omega_N$ up to rescaling the sets.
    
	The general aim of this note is to prove some results similar to those already present in the literature for more specific functionals. What connects these two problems is a common technique: the competing terms enjoy some quantitative stability inequalities, and we combine them to obtain the rigidity results for the minimizers in some parameters ranges.

    The first section introduces the notation and some preliminary results that ease some computations. We gather here those quantities that play a significant role in both of the problems (with minor differences), while the more specific ones are presented in the successive sections.
	
	In the second section we focus on the ``perimeter problem'', that is already studied in \cite{km1,km2,boncris,franklieb} with the standard perimeter in place of $\perim_s$ and a negative power as function $g$. The problem with the fractional perimeter was studied in \cite{f2m3}, where the negative power varies in the largest possible range. We make use of many tools developed in this last article to handle the term involving $\perim_s$. The two more recent papers \cite{novaga-pratelli, cfp} instead treat the problem \eqref{eq:FF} without assuming $g$ to be homogeneous. In this paper we apply some results contained in \cite{cfp}, combined with a strong version of the quantitative isoperimetric inequality present in \cite{f2m3}, to prove the following theorem:
	\begin{mainthm}\label{mainthm:FF}
		Let $s\in(0,1)$ and let $g:\R^N\setminus\{0\}\to[0,+\infty)$ be a function satisfying \HH. There exists $\gamma_{0}(s,g)>0$ such that, if $\gamma<\gamma_{0}$, then any minimizer $E_{\gamma}$ of $\FF_{\gamma}$ with volume $\omega_{N}$ coincides with a ball of radius $1$.
	\end{mainthm}
	We point out that other generalizations are being studied, and they concern both of the terms of $\FF_{\gamma}$. For example in \cite{alama2020nonlocal} the authors establish some existence and regularity result for the minimizers of a functional containing a weighted perimeter (but a repulsive term of negative-power type). Moreover, if the weight is a monomial function, they also recover that the balls are the unique minimizers in the reasonable volume regime (that is, where the attractive term should heuristically be stronger than the repulsive one). Another example is Pegon's article \cite{pegon}, where the author treats the problem with the standard perimeter and a radial and globally integrable repulsive kernel. The global integrability assumption allows Pegon to rewrite the minimization problem as the difference between the perimeter and a generalized non-local perimeter.
	
	The third section is devoted to the problem \eqref{eq:GG}, where we show how the proofs present in \cite{franklieb19} can be modified to work with a quite generic function $g$. Here we need to impose an additional assumption on $g$, which is very close to the setting in which Frank and Lieb assert that their theorems could be generalized. Indeed, they also include a non-homogeneous attractive term in their generalization, but we preferred to deal only with the other one because it is more similar to the problem \eqref{eq:FF}. In the end we are able to prove the following theorem, that is analogous to \autoref{mainthm:FF}, where we show that the attractive term completely overwhelms the repulsive one if the volume constraint is \textit{big} enough:
	\begin{mainthm}\label{mainthm:GG}
		Let $\alpha>0$ and let $g:\R^N\setminus\{0\}\to[0,+\infty)$ be a function satisfying \HI. There exists a threshold $m_{0}=m_{0}(\alpha,g,N)>0$ such that the only minimizers for $\GG$ with ``volume" constraint $m>m_{0}$ are the characteristic function of balls.
	\end{mainthm}
	We point out that this kind of problems can be posed also in the measure setting (instead of $L^1(\R^N;[0,1])$) and one can impose a generic mass constraint (i.e. not necessarily large), and in this more general situation many different phenomena can appear. They are studied for example in \cite{BCLR} and \cite{CDM}, where even local minimizers are considered (with respect to a certain Wasserstein distance). Many numerical experiments have been performed, which give an idea of the complex situation that arises from this relatively simple functional. We also highlight that in \cite{CDM} the restrictions on the interaction kernel allow the authors to answer some regularity questions by means of the obstacle problem theory, while we work essentially by hands, and this permits us to consider a quite generic function $g$.
    \subsection*{Acknowledgements} The content of this paper is part of the Master Thesis of the author, who is very grateful to Aldo Pratelli for his supervision and his guidance during the development of this project.
	
	\section{Notation and preliminary results}\label{sec:preliminary}
        We will denote respectively by $B, B(x,r), B[m]\subset\R^N$ the unitary ball centered in the origin, the ball of radius $r$ centered in $x$ and the ball centered in $0$ with $|B[m]|=m$. Also, if $l>0$ and $x\in\R^d$, we denote the $d$-dimensional cube centered in $x$ with sides of length $2l$ as
        \[
            Q_d(x,l)\coloneqq \left\{y\in\R^d\st|x_i-y_i|\leq l\ \forall i=1,\ldots,d\right\}.
        \]
        Given any function $f:\R^N\to\R$ and a parameter $t>0$, we define the dilated function $f[t](x)\coloneqq f(x/t)$, so that $\int f[t](x)\dx = t^N\int f(x)\dx$.
        \subsection*{Hypotheses on $g$} Here we gather the hypotheses that our setting requires. The first one contains the most general assumptions, which will be almost always supposed to hold, while the second one will be used only in \autoref{sec:2}. We suppose to have $g:\R^N\setminus\{0\}\to[0,+\infty)$ such that\\
        \begin{itemize}
            \item[\HH] $g\in L^1_{loc}(\R^N)$, $g(tx)\leq g(x)$ for every $t\geq1$ and for every $x\neq 0$ and moreover there exists $R_g>0$ such that $g(x)$ is bounded in $\{|x|>R_g\}$;
            \item[\HI] the function $g$ satisfies the condition \HH\ and moreover it is radial and the function $x\mapsto |x|g(x)$ is of class $L^1_{loc}(\R^N)$.
        \end{itemize}

        It is not hard to see that the local integrability is necessary in order to ensure that $\FF_{\gamma}$ (respectively $\GG$) is finite on $B(x,r)$ for every $r>0$ (respectively on the characteristic function of $B(x,r)$). Therefore, the hypotheses in \HH\ are very mild, while those in \HI\ are there primarily to have some good regularity property for the potential that is defined in \eqref{eq:v_E-RR}.\\
        Since we will frequently integrate $g$ onto sets, we define the repulsive potential generated by a generic set $E\subset\R^N$ and its own repulsive energy respectively by
        \begin{equation}\label{eq:v_E-RR}
            v_E(x)\coloneqq \int_Eg(x-y)\dy,\qquad \RR(E)\coloneqq\int_E\int_Eg(x-y)\dx\dy.
        \end{equation}
        If instead we consider a function $h\in L^1\cap L^{\infty}$, we denote by $v_h$ and $\RR(h)$ respectively the potential generated by $h$ and its own interaction energy, and they are defined with formulas analogous to \eqref{eq:v_E-RR}. When we consider the problem \eqref{eq:GG}, we also denote the attractive part of $\GG(h)$ by $\mathcal{I}_{\alpha}(h)$, where $\alpha>0$ is a fixed parameter. Sometimes we will also compute the interaction between different sets or functions, which we denote by
        \begin{gather*}
            \RR(E_1,E_2)=\int_{E_1}\int_{E_2}g(x-y)\dx\dy,\\
            \mathcal{I}_{\alpha}(h_1,h_2)=\iint |x-y|^{\alpha}h_1(x)h_2(y)\dx\dy,\qquad \RR(h_1,h_2)=\iint g(x-y)h_1(x)h_2(y)\dx\dy,
        \end{gather*}
        where $E_1,E_2\subset\R^N$ are sets with finite volume and $h_1,h_2\in L^1\cap L^{\infty}$ are functions with bounded support (but we do not suppose that they are non-negative). Similarly, $\GG(h_1,h_2)$ denotes the full interaction energy between $h_1$ and $h_2$. Moreover, it will be useful to express in a compact way the repulsive potential of a ball computed at a certain distance from the origin (in the case of $g$ radial). To this end, we define the function
        \begin{equation}\label{eq:psi}
            \psi(R,r)\coloneqq \int_{B(0,R)}g(re_1-y)\dy.
        \end{equation}
        For the second problem, it is also useful to give a name to the functions representing the attractive potential of a unitary ball and the \textit{full} potential that is produced by a ball of volume $m$. Namely, given $m>0$, we define the functions
        \begin{equation}\label{eq:potentials}
        \begin{split}
            \varphi(r)&\coloneqq \int_{B}|re_{1}-y|^{\alpha}\dy,\\
            \Phi_{m}(r) &\coloneqq \int_{B[m]}\left(|re_{1}-y|^{\alpha}+g(re_{1}-z)\right)\dy
        \end{split}
        \end{equation}
        for all $r\geq 0$.\\
        In isoperimetric problems it is very important to deal with the \textit{asymmetry} of a set $E$, that we denote by $A(E)$ and that is defined as
        \[
            A(E)\coloneqq \inf_{x\in\R^{N}}\frac{|E\symmdiff (x+B[m])|}{m} \qquad\text{where }m=|E|.
        \]
        One can easily see that the $\inf$ is attained using a compactness argument. There exists also an analogous quantity defined for functions $h\in\mathcal{K}$:
        \[
            A(h) \coloneqq \frac{\inf\left\{\norm{h-\Chi{x+B[m]}}_{L^{1}}\st x\in\R^{N}\right\}}{m}\qquad \text{where }m=\norm{h}_{1},
        \]
        and again the $\inf$ is attained. Of course this quantity makes sense also for $h\in L^1$, but it is significant only for $h$ positive and bounded.\\
        Many times during our computations we will not track down the precise constants appearing, and we will often denote with the same letter or expression a constant that changes from a line to the other. This is done in order to keep the formulas shorter, and it is also justified by the fact that many of those constants are probably not sharp, as it is also pointed out in \cite[Remarks 2]{franklieb19}.
        \begin{lemma}\label{lemma:lipschitz-RR}
            Let $g:\R^N\setminus\{0\}\to[0,+\infty)$ be a function satisfying \HH. If $h_1,h_2\in\mathcal{K}$ are two functions with $\max\{\norm{h_1}_1,\norm{h_2}_1\}\leq \widetilde{m}<+\infty$, then
            \[
                |\RR(h_1)-\RR(h_2)|\leq C(g,\widetilde{m})\norm{h_1-h_2}_1.
            \]
            Moreover, if $\norm{h}_1\geq \omega_N$, then $v_h\leq C(g)\norm{h}_1$ everywhere.
        \end{lemma}
        \begin{proof}
            It is very easy to see that the function $x\mapsto \int g(x-y)h(y)\dy$ is bounded, and more precisely there exists a constant $C_1(g,\norm{h}_1)<+\infty$ that controls its $L^{\infty}$ norm:
            \begin{equation}\label{eq:bound-v}
                \begin{split}
                v_h(x)&=\int g(x-y)h(y)\dy= \int_{B(x,R_g)}g(x-y)h(y)\dy+\int_{\R^N\setminus B(x,R_g)}g(x-y)h(y)\dy\\
                &\leq \int_{B(x,R_g)}g(x-y)\dy+\int_{\R^N\setminus B(x,R_g)}g(x-y)h(y)\dy\\
                &\leq \int_{B(0,R_g)}g(y)\dy+ \norm{h}_1\sup\left\{g(x):|x|>R_g\right\} \eqqcolon C_1(g,\norm{h}_1),
                \end{split}
            \end{equation}
            where we used that $g$ is locally integrable and it is bounded outside $B(0,R_g)$. Notice also that $C_1(g,m)$ is increasing as a function of $m$.\\
            Then the statement follows immediately from the ``linearity'' of $\RR(h_1,h_2)$ in both of the arguments, separating the contributions of $(h_1 - h_2)\vee 0$, $(h_2- h_1)\vee 0$ and $h_1\wedge h_2$:
            \begin{equation}\label{eq:RR-difference-identity}
            \begin{split}
                |\RR(h_1)-\RR(h_2)| &= |\RR((h_1 - h_2)\vee 0)+2\RR((h_1 - h_2)\vee 0,h_1\wedge h_2)+\RR(h_1\wedge h_2)\\
                &\qquad -\RR((h_2- h_1)\vee 0)-2\RR((h_2- h_1)\vee 0,h_1\wedge h_2)-\RR(h_1\wedge h_2)|\\
                &\leq |\RR((h_1 - h_2)\vee 0)+2\RR((h_1 - h_2)\vee 0,h_1\wedge h_2)|\\
                &\qquad+|\RR((h_2- h_1)\vee 0)+2\RR((h_2- h_1)\vee 0,h_1\wedge h_2)|.
            \end{split}
            \end{equation}
            We bound only the first term in the last expression, since the second one is analogous:
            \begin{align*}
                &\RR((h_1 - h_2)\vee 0)+2\RR((h_1 - h_2)\vee 0,h_1\wedge h_2)\\
                &\qquad=\iint g(x-y)\left[(h_1(y) - h_2(y))\vee 0)+2(h_1(y)\wedge h_2(y))\right](h_1(x) - h_2(x))\vee 0)\dy\dx\\
                &\qquad\leq \int3 v_{h_1}(x)((h_1(x) - h_2(x))\vee 0)\dx\\
                &\qquad\leq 3C_1(g,\widetilde{m})\norm{h_1-h_2}_1.
            \end{align*}
            With this we proved the first bound, and now we easily obtain the second one dividing the integral as we did in \eqref{eq:bound-v} and using the local integrability of $g$:
            \begin{align*}
                v_h(x)&=\int g(x-y)h(y)\dy= \int_{B(x,R_g)}g(x-y)h(y)\dy+\int_{\R^N\setminus B(x,R_g)}g(x-y)h(y)\dy\\
                &\leq \int_{B(0,R_g)}g(y)\dy + \norm{h}_1\sup\left\{g(x):|x|>R_g\right\}\\
                &\leq C(g)\frac{\norm{h}_1}{\omega_N}+\norm{h}_1\sup\left\{g(x):|x|>R_g\right\} = C(g)\norm{h}_1.
            \end{align*}
        \end{proof}
        \begin{rem}\label{rem:lipschitz-RR}
            One can specialize the previous statement to $h_2=\Chi{B[m]}$ with $m=\norm{h_1}_1$ to obtain on the right hand side the asymmetry (since the left hand side is translation invariant).
        \end{rem}

\section{Fractional perimeter as attractive term}\label{sec:1}
	We firstly recall the definitions of fractional perimeter and fractional Sobolev norm, together with some other important classes of sets that we will make use of. Successively we present the results that we are going to use and that have been already developed in \cite{f2m3,cfp}.
	\begin{defn}
		The fractional perimeter of order $s\in(0,1)$ is denoted by $\perim_s$ and is defined as
		\[
            \perim_s(E)\coloneqq \int_E\int_{E^c}\frac{1}{|x-y|^{N+s}}\dx\dy
        \]
		for every measurable set $E\subset \R^N$ (of course, it could possibly be $+\infty$).
	\end{defn}
	\begin{defn}\label{defn:frac-sobolev-seminorm}
        Given an open set $\Omega\subset\R^N$ and $u:\Omega\to\R$, its fractional Sobolev seminorm of order $s$ (and exponent $2$) is defined as
		\[
            [u]_s\coloneqq \left(\int_{\Omega}\int_{\Omega}\frac{|u(x)-u(y)|^2}{|x-y|^{N+2s}}\dx\dy\right)^{1/2}.
        \]
		Thus, we can also define the fractional Sobolev norm as $\norm{u}^2_{W^{s,2}}\coloneqq \norm{u}^2_{L^2(\Omega)}+[u]^2_s$.\\
		Moreover, we will use an analogous definition if $M^m\subset \R^N$ is a compact $m$-submanifold embedded in $\R^N$: given a function $u:M\to\R$, we define its fractional Sobolev seminorm as
		\[
            [u]_s\coloneqq \left(\int_{M}\int_{M}\frac{|u(x)-u(y)|^2}{|x-y|^{m+2s}}\diff \haus^m(x)\diff\haus^m(y)\right)^{1/2},
		\]
		where $|x-y|$ is the euclidean distance between $x$ and $y$ measured in $\R^N$. As before, we define $\norm{u}^2_{W^{s,2}}\coloneqq \norm{u}^2_{L^2(M;\haus^m)}+[u]^2_{s}$. In order to simplify the notation, we will often omit the set where we compute the various norms/seminorms when it coincides with the domain of the function $u$.
	\end{defn}
    \begin{rem}
    	From the definitions it is clear that $2\perim_s(E) =\left[\Chi{E}\right]_{s/2}^2$ with $\Omega=\R^N$ in \autoref{defn:frac-sobolev-seminorm}.
    \end{rem}
    The next definition appears in \cite{fuglede,cicleo}, and it is important in our study since we will use an inequality that is closely related to the so-called Fuglede inequality for nearly spherical sets. The $W^{1,\infty}$ bound in our definition is different from the one present in the aforementioned papers because \autoref{thm:fuglede-frac} already contains the suitable bound for the Sobolev norm.
    \begin{defn}\label{defn:nearly-spherical}
    	An open set $E\subset \R^N$ is \textit{nearly spherical} if $|E|=\omega_N$, its barycenter is $0$ and there exists a Lipschitz function $u:\bdry B\to(-1,1)$ such that
    	\[
            \bdry E = \{(1+u(x))x\st x\in\bdry B\},
        \]
    	with $\norm{u}_{\infty}+\norm{\grad u}_{\infty}\leq1$.
    \end{defn}
    The following are two different versions of the quantitative isoperimetric inequality for the fractional perimeter that have been developed in \cite{f2m3} (respectively labelled as Theorem 1.1 and Theorem 2.1 in that paper) and we report them here for convenience. The first is the fractional counterpart of the general isoperimetric inequality, and it is remarkable that the asymmetry appears at the power $2$ just like in the standard one (see \cite[Theorem 1.1]{fumapra}), while the analogue of \autoref{thm:fuglede-frac} for the standard perimeter can be found in \cite[Theorem 4.1]{cicleo}.
    
	\begin{thm}\label{thm:quantitative-i-i-frac}
		Let $N\geq 2$ and $s\in (0,1)$. There exists a constant $C(N,s)>0$ such that, for every $E\subset\R^{N}$ with finite measure, it holds that
		\[
			\perim_{s}(E)\geq \perim_{s}(B[m]) + C(N,s)m^{(N-s)/N}A(E)^{2},
		\]
		where $m=|E|$.
	\end{thm}
	\begin{thm}\label{thm:fuglede-frac}
		There exists $\delta_{0}<1/2$ and $C_{0}>0$ that depend only on $N$ with the following property: if $E\subset\R^{N}$ is a nearly spherical set with $\norm{u}_{W^{1,\infty}(\bdry B)}<\delta_{0}$, then
		\[
			\perim_{s}(E)-\perim_{s}(B) \geq C_{0}\left([u]_{\frac{1+s}{2}}^{2}+s\perim_{s}(B)\norm{u}^{2}_{L^{2}(\bdry B)}\right)	\qquad\forall s\in(0,1).
		\] 
	\end{thm}
    Finally, we rewrite here for convenience the statement of \cite[Lemma 4.5]{f2m3}, which is useful to cut a set with a good control on the fractional perimeter of the new set.
    \begin{lemma}\label{lemma:cutting}
		Let $N\geq 2$ and $s\in(0,1)$. Given a set $E\subset \R^{N}$ such that $|E\setminus B|\leq \eta<1$, there exists a radius $1\leq r(E)\leq 1+C(N,s)\eta^{1/N}$ such that
		\[
            \perim_{s}(E\cap B(0,r(E))) \leq \perim_{s}(E)-C(N,s)\frac{|E\setminus B(0,r(E))|}{\eta^{s/N}}.
        \]
	\end{lemma}
	\begin{rem}\label{rem:monotonicity-frac-per}
		We observe that this minimization problem is monotone with respect to the mass constraint: given $E\subset \R^N$ with mass $m+h$ we can always cut it with an hyperplane in such a way that the new set $G$ satisfies $|G|=m$. But then $\FF_{\gamma}(G)\leq \FF_{\gamma}(E)$ since the repulsive term is clearly reduced (the interaction kernel is non-negative) and also the fractional perimeter is decreased thanks to \cite[Lemma B.1]{f2m3}.
	\end{rem}
	We state now the basic existence theorem, that uses \autoref{thm:quantitative-i-i-frac} and \autoref{lemma:cutting} to deal with the perimeter term and combines them with \autoref{lemma:lipschitz-RR} which provides a good control on the repulsive term. Our proof of \autoref{thm:existence-frac} takes some ideas from the proof of \cite[Lemma 5.1]{f2m3}, but it is simpler since we do not track precisely the dependence of the various constants that appear.
	%
	\begin{thm}\label{thm:existence-frac}
		Given $s\in(0,1)$ and a function $g:\R^N\setminus\{0\}\to[0,+\infty)$ satisfying \HH, there exists $\gamma_{0}(N,s,g)>0$ such that, if $\gamma<\gamma_{0}$, then $\FF_{\gamma}$ admits a minimizer with volume constraint $\omega_{N}$. 
	\end{thm}

	\begin{proof}
        We will prove that, if $\gamma$ is small enough, then we find some candidates for the problem \eqref{eq:FF} that are bounded sets. Then we will apply a compactness result to conclude via the standard method of calculus of variations.\\
        We prove the boundedness of some candidates for \eqref{eq:FF}, and to this aim we can suppose that the asymmetry of a given candidate is non-zero (otherwise this step is not necessary) and it can be taken as small as we want. In fact, let $E$ be a competitor for the minimization problem with volume $\omega_{N}$ and parameter $\gamma$, then we can suppose that $\FF_{\gamma}(E)\leq \FF_{\gamma}(B)$, and thus using \autoref{thm:quantitative-i-i-frac} and \autoref{rem:lipschitz-RR} we have that
		\begin{equation}\label{eq:inutile18}
			C(N,s)A(E)^{2}\leq \perim_{s}(E)-\perim_{s}(B) \leq \gamma(\RR(B)-\RR(E)) \leq \gamma C(g,N)A(E).
		\end{equation}
		Hence we can take $2\omega_N\gamma_{0}<C(N,s)/C(g,N)$ so that every set $E$ chosen as before satisfies $|E\setminus B|<1/2$, and thanks to the translation invariance of the problem we can suppose that the optimal ball for the asymmetry is centered in the origin.\\
		Now we can use \autoref{lemma:cutting} and see that there exists a radius $r(E)$, that satisfies $1\leq r(E)\leq 1+C(N,s)|E\setminus B|^{1/N}$, with the following property:
		\[
            \perim_{s}(E\cap B(0,r(E)))\leq \perim_{s}(E)-C(N,s)\frac{|E\setminus B(0,r(E))|}{\eta^{s/N}}\qquad\text{with }\eta=|E\setminus B|.
        \]
		For the next computations we define the set $E'=E\cap B(0,r(E))$ and the parameter $\lambda = (|E|/|E'|)^{1/N}$. Using the rescaling inequalities for the fractional perimeter and the repulsive term (see \cite{novaga-pratelli}) we arrive to
		\begin{equation}\label{eq:inutile17}
			\FF_{\gamma}(\lambda E')\leq \lambda^{N-s}\perim_{s}(E')+\gamma\lambda^{2N}\RR(E')\leq \lambda^{2N}\FF_{\gamma}(E').
		\end{equation}
		We now define $p=|E\setminus E'|/|E|$, so $\lambda=1/(1-p)^{1/N}$ and, reducing $\gamma_{0}$ if necessary, we can suppose that $p\leq 1/2$. Thus, using that the function $p\mapsto(1-p)^{-2}$ is Lipschitz in $[0,1/2]$, we obtain the estimate $\lambda^{2N}\leq1+Cp$ for a universal constant $C>0$. Hence we can plug it into \eqref{eq:inutile17} and use the monotonicity of $\RR(E)$ with respect to the inclusion (see \autoref{rem:monotonicity-frac-per}) to obtain
		\begin{align*}
			\FF_{\gamma}(\lambda E') &\leq (1+Cp)\FF_{\gamma}(E')\leq \FF_{\gamma}(E)-C(N,s)\frac{|E|p}{\eta^{s/N}}+Cp\FF_{\gamma}(E')\\
			&\leq\FF_{\gamma}(E) + p\left(C\FF_{\gamma}(E)-C(N,s)\frac{\omega_{N}}{\eta^{s/N}}\right).
		\end{align*}
		We can suppose $\gamma_0<1$, so that $\FF_{\gamma}(E)\leq \FF_{\gamma}(B) \leq \FF_1(B) \eqqcolon C(s,g)$, and we can rewrite the previous inequality as
		\[
            \FF_{\gamma}(\lambda E') \leq \FF_{\gamma}(E) + p\left(C(s,g)-\frac{C(N,s)}{\eta^{s/N}}\right).
        \]
		Exploiting again \eqref{eq:inutile18} we see that, if $\gamma_0$ is small enough, then $\eta^{s/N}<C(N,s)/C(s,g)$. With this choice of $\gamma_0$ we see that $\FF_{\gamma}(\lambda E')<\FF_{\gamma}(E)$ since we supposed that $A(E)>0$. The new set $\lambda E'$ is contained in $B(0,r(E))$ and $r(E)\leq 2+2C(N,s)$ thanks to the bound on $r(E)$ provided by \autoref{lemma:cutting}. This argument guarantees that, without loss of generality, we can consider only competitors that are contained in a fixed ball.\\
        Now we get the existence result. Let $E_k$ be a minimizing sequence for $\FF_{\gamma}$ with constrained mass $\omega_N$ such that $E_{k}\subset B(0,2+2C(N,s))$. Then we can use the compact embedding theorem for fractional Sobolev spaces (we refer to \cite[Theorem 7.1]{dinpava}) to apply the standard method of calculus of variations: the compact embedding theorem provides a subsequence (not relabelled) that converges in $L^1_{loc}$ topology, then both $\perim_{s}$ and $\RR$ are lower semicontinuous with respect to the $L^{1}_{loc}$ convergence thanks to Fatou's lemma, and finally the mass constraint is preserved by that convergence since the sets $E_k$ are uniformly bounded.
	\end{proof}
	\begin{defn}\label{defn:min-frac-per}
        Given $C>0$, we say that a Borel set $E\subset\R^{N}$ is a $C$\textit{-minimizer of the} $s$\textit{-perimeter} if for every bounded set $F\subset\R^{N}$ it holds that
        \[
            \perim_{s}(E)\leq \perim_{s}(F)+C|E\symmdiff F|.
        \]
    \end{defn}
    We prove that minimizers of $\FF_1$ are $C$\textit{-minimizers} of the $s$-perimeter for some $C>0$, and this will be useful because of the good regularity properties held by those sets. We refer to \cite[Chapter 21]{maggi} and \cite{tamanini84} for the standard perimeter case, where the classical regularity theory is also developed. Here we write the result that we will use, which is a particular version of \cite[Corollary 3.6]{f2m3}:
    \begin{thm}\label{thm:reg-almost-min-frac-per}
    	If $N\geq 2$, $C\geq0$, $s\in(0,1)$, $E_k\subset\R^N$ is a $C$-minimizer of the $s$-perimeter for every $k\in\N$ and $E_k\to B$ in $L^1$, then there exists $\alpha\in(0,1)$ and a sequence $u_k\in C^{1,\alpha}(\bdry B)$ such that
    	\begin{enumerate}
            \item for $k$ large enough $\bdry E_k=\left\{(1+u_k(x))x\st x\in\bdry B\right\}$;
            \item $\lim_{k\to\infty}\norm{u_k}_{C^{1,\alpha}}=0$.
    	\end{enumerate}
    \end{thm}
	\begin{prop}\label{prop:almost-min-frac-per}
		Let $E\subset\R^{N}$ be a minimizer of $\FF_{\gamma}$ with measure $\omega_N$. Then $E$ is a $C$-minimizer of the $s$-perimeter for some constant $C=C(\gamma,g,s)>0$.
	\end{prop}

	\begin{proof}
		Suppose by contradiction that there exists a sequence of sets $F_{k}\subset \R^{N}$ with $\perim_{s}(F_{k})\leq \perim_{s}(E)$, $|E\symmdiff F_{k}|\neq 0$ and
		\[
            C_{k}\coloneqq\frac{\perim_{s}(E)-\perim_{s}(F_{k})}{|E\symmdiff F_{k}|}\to +\infty.
        \]
		Using the isoperimetric inequality for the fractional perimeter and that $\perim_{s}(F_{k})\leq \perim_{s}(E)$ we have that $|F_{k}|$ is bounded by a constant that depends on $N$, $g$ and $s$. Now we can estimate $\FF_{\gamma}(F_{k})$ as
		\begin{equation}\label{eq:inutile8}
			\begin{split}
				\FF_{\gamma}(F_{k})&=\perim_{s}(E)-C_{k}|E\symmdiff F_{k}| + \gamma \RR(F_{k})\\
				&=\perim_{s}(E)-C_{k}|E\symmdiff F_{k}| + \gamma (\RR(F_k)-\RR(E)+\RR(E))\\
				&\leq \FF_{\gamma}(E)-C_{k}|E\symmdiff F_{k}|+\gamma C(g)|E\symmdiff F_{k}|,
			\end{split}
		\end{equation}
		where we used \autoref{rem:lipschitz-RR} in the last inequality. Since $C_k\to+\infty$, we notice that \eqref{eq:inutile8} implies that $|E\symmdiff F_k|\to0$. Now let us take $k$ so large that $C_k>\gamma C(g)$. Then $|F_k|<|E|$: the assumption on $C_k$ guarantees that $\FF_{\gamma}(F_k)<\FF_{\gamma}(E)$, and if $|F_k|\geq |E|$ then we can cut $F_k$ with an hyperplane to obtain a new set $F_k'$ with $|F_k'|=|E|$. Then $\FF_{\gamma}(F_k')\leq \FF_{\gamma}(F_k)<\FF_{\gamma}(E)$, but this is not possible since $E$ is a minimizer of $\FF_{\gamma}$.
        Now that we are sure to have $|F_k|<|E|$ we can rescale the sets $F_{k}$ in order to have the right measure. Notice that $|F_{k}|=|E|+|F_{k}\setminus E|-|E\setminus F_{k}|$, hence we can define
		\begin{equation}\label{eq:lambda_k}
			\lambda_{k} = \left(\frac{|E|}{|F_{k}|}\right)^{1/N} = \left(\frac{|E|}{|E|+|F_{k}\setminus E|-|E\setminus F_{k}|}\right)^{1/N}\leq \left(1-\frac{|E\symmdiff F_{k}|}{|E|}\right)^{-1/N},
        \end{equation}
		and we have already noticed in the proof of \autoref{thm:existence-frac} that $\FF_{\gamma}(\lambda_{k}F_{k})\leq \lambda_{k}^{2N}\FF_{\gamma}(F_{k})$. If we combine this estimate with \eqref{eq:inutile8}, we can take $k$ large enough to have that $C_{k}>2\gamma C(g)$ and Taylor expand the rightmost formula in \eqref{eq:lambda_k} with $|E\symmdiff F_{k}|\ll1$ to get
		\begin{align*}
			\FF_{\gamma}(\lambda_{k}F_{k})&<\left(1-\frac{|E\symmdiff F_{k}|}{|E|}\right)^{-2}\left(\FF_{\gamma}(E)-C_{k}|E\symmdiff F_{k}|+\gamma C(g)|E\symmdiff F_{k}|\right)\\
			&\leq\left(1-\frac{|E\symmdiff F_{k}|}{|E|}\right)^{-2}\left(\FF_{\gamma}(E)-\frac{C_{k}}{2}|E\symmdiff F_{k}|\right)\\
			&\leq \FF_{\gamma}(E)-\frac{C_{k}}{2}|E\symmdiff F_{k}| + C(N)|E\symmdiff F_{k}|\FF_{\gamma}(E).
		\end{align*}
		And from this last inequality we arrive to a contradiction since $C_{k}$ is going to $+\infty$ as $k\to\infty$, so $\FF_{\gamma}(\lambda_{k}F_{k})<\FF_{\gamma}(E)$. In the end, notice that the threshold for $C_{k}$ depends only on $\gamma$, $g$ and $s$.
	\end{proof}

	\begin{rem}
		Looking at the proof it is immediate to notice that if we consider a minimizer $E'$ for $\FF_{\gamma'}$ with $|E'|=\omega_N$ and $\gamma'\leq\gamma$, then $E'$ is a $C$-minimizer of the $s$-perimeter with a constant $C = C(\gamma,g,s)$.
	\end{rem}

	Here we prove a simple growth property in $0$ held by the function $g$:
	\begin{lemma}\label{lemma:growth-g}
		If $g:\R^N\setminus\{0\}\to[0,+\infty)$ is a radial function satisfying \HH, then there exists a constant $C(g)>0$ such that
		\[
			g(x) \leq \frac{C(g)}{|x|^{N}}\qquad\forall x\in B\setminus\{0\}.
		\]
		More precisely, we must have that $\displaystyle\limsup_{x\to0}g(x)|x|^{N}= 0$.
	\end{lemma}

	\begin{proof}
		We argue by contradiction. Suppose that there exists a sequence $r_{k}\to 0^{+}$ such that $\limsup_{k}g(r_{k}e_1)r_{k}^{N} = \lim_kg(r_{k}e_1)r_{k}^{N}>0$. Without loss of generality we can assume that $r_{k}<1$ and $r_{k+1}<r_{k}/2$ for all $k\in\N$. And then the monotonicity of $g$ implies that
		\begin{align*}
			\int_{B}g(x)\dx &\geq \omega_{N}\sum_{k=1}^{+\infty}g(r_{k}e_1)(r_{k}^{N}-r_{k+1}^{N})\\
			&\geq \omega_{N}\sum_{k=1}^{+\infty}g(r_{k}e_1)r_{k}^{N}\left(1-\frac{1}{2^{N}}\right)= C_{N}\sum_{k=1}^{+\infty}g(r_{k}e_1)r_{k}^{N}.
		\end{align*}
		Since $g\in L^{1}(B)$ we have that the last series converges, so its terms have to be infinitesimal, but this is not compatible with the fact that $\lim_{k}g(r_{k}e_1)r_{k}^{N}>0$.\\
		We proved only the second part of the statement, but the first part can be proved reasoning in an analogous way. Indeed, it is sufficient to take two sequences $r_{k}\in (0,1)$ and $C_{k}\to+\infty$ with $g(r_{k}e_{1})r_{k}^{N}>C_{k}$. Then notice that $r_{k}$ must converge to $0^{+}$ (otherwise we would reach immediately a contradiction with the integrability of $g$), so that the previous argument works again.
	\end{proof}
    
    \begin{prop}\label{prop:frac-sobolev-bound-R}
        Let $E\subset \R^N$ be a nearly spherical set, with $\bdry E$ parametrized by $u:\bdry B\to(-1,1)$ according to \autoref{defn:nearly-spherical}. If $\norm{u}_{\infty}\leq 1/4$ and $g:\R^N\setminus\{0\}\to[0,+\infty)$ is a radial function satisfying \HH, then
        \begin{equation}\label{eq:R-bound-frac}
            \RR(B)-\RR(E)\leq C(g)\norm{u}_{W^{r,2}(\bdry B)}^{2}\qquad\forall r\in[1/2,1).
        \end{equation}
    \end{prop}
    \begin{proof}
        This result can be proved following the procedure exploited in the proof of \cite[Theorem A]{cfp}. There, the inequality \eqref{eq:R-bound-frac} is proved with the standard $W^{1,2}$ norm in place of the fractional one, but following that proof one notices that the $L^2$ norm of $\grad u$ pops up only because of the following inequality:
        \begin{equation*}
        	\int_{\bdry B}\int_{\bdry B}g(x-y)|u(x)-u(y)|^2\dx\dy \leq C(g)\norm{u}_{W^{1,2}(\bdry B)}^2.
        \end{equation*}
        If we are able to write another inequality with the fractional Sobolev norm, then we can follow exactly the proof of \cite[Theorem A]{cfp}. In fact, we can do that: thanks to \autoref{lemma:growth-g}, for any $r_0\in[0,1)$ we have that
    	\begin{align*}
            \int_{\bdry B}\int_{\bdry B}g(x-y)|u(x)-u(y)|^2\dx\dy&\leq C(g)\int_{\bdry B}\int_{\bdry B}\frac{|u(x)-u(y)|^2}{|x-y|^N}\dx\dy\\
            &\leq 2^{r_0}C(g)\int_{\bdry B}\int_{\bdry B}\frac{|u(x)-u(y)|^2}{|x-y|^{N+r_0}}\dx\dy \leq 2C(g)[u]^2_{\frac{1+r_0}{2}},
        \end{align*}
        where clearly the fractional Sobolev seminorm is relative to the hypersurface $\bdry B\subset \R^N$. This is the desired inequality since we can take $r=(1+r_0)/2$.
    \end{proof}
    \begin{rem}
    	Notice that the above inequality is stronger than the one used in \cite{cfp} because of \cite[Proposition 2.2]{dinpava}, that can be applied thanks to the compactness of $\bdry B$.
    \end{rem}


	\begin{proof}[Proof of \autoref{mainthm:FF}]
		We will determine later the value of $\gamma_0<1$, but for now let us suppose to have fixed it. For any $\gamma<\gamma_0$ let $E_{\gamma}\subset\R^{N}$ be a minimizer of $\FF_{\gamma}$ with mass $\omega_{N}$ and barycenter in $0$. From \autoref{prop:almost-min-frac-per} we know that they are $C$-minimizers of the $s$-perimeter for some $C=C(\omega_{N},g,s)>0$. We can apply \autoref{thm:quantitative-i-i-frac} and \autoref{rem:lipschitz-RR} to see that
		\[
			C(N,s)A(E_{\gamma})^{2}\leq \perim_{s}(E_{\gamma})-\perim_{s}(B) \leq \gamma(\RR(B)-\RR(E_{\gamma})) \leq \gamma C(g)A(E_{\gamma}).
		\]
		From this inequality we deduce that $A(E_{\gamma})\to0$ as $\gamma\to0$, or equivalently that $E_{\gamma}\to B$ in $L^{1}$.\\
		Now we can apply \autoref{thm:reg-almost-min-frac-per} to see that, for $\gamma$ small enough, the sets $E_{\gamma}$ are nearly spherical (see \autoref{defn:nearly-spherical}), with $\bdry E_{\gamma}$ parametrized by $u_{\gamma}$. Moreover, the family of functions $\{u_{\gamma}\}_{\gamma\in(0,1)}$ is bounded in $C^{1,\alpha}(\bdry B)$ for some $\alpha\in (0,1)$ and $\norm{u_{\gamma}}_{C^{1}(\bdry B)}\to 0$.\\
		Up to reducing again $\gamma_{0}$, we can suppose that the hypotheses of \autoref{thm:fuglede-frac} are fulfilled for $E_{\gamma}$ with $\gamma<\gamma_{0}$, and thus the following chain of inequalities holds:
		\begin{equation*}
			C_{0}\left([u_{\gamma}]_{\frac{1+s}{2}}^{2}+s\perim_{s}(B)\norm{u_{\gamma}}^{2}_{L^{2}(\bdry B)}\right)\leq \perim_{s}(E_{\gamma})-\perim_{s}(B) \leq \gamma(\RR(B)-\RR(E_{\gamma})).
		\end{equation*}
 		Now we can take $\gamma_0$ so small that $\norm{u_{\gamma}}_{\infty}\leq 1/4$ and continue that chain of inequalities applying \autoref{prop:frac-sobolev-bound-R} with $r=(1+s)/2$:
%
		\[
			C(N,s)\left([u_{\gamma}]^{2}_{\frac{1+s}{2}}+\norm{u_{\gamma}}_{2}^{2}\right) \leq \gamma \cdot 2C(g)\left([u_{\gamma}]^{2}_{\frac{1+s}{2}}+\norm{u_{\gamma}}_{2}^{2}\right).
		\]
		If $\gamma<C(N,s)/(2C(g))$ we have that the above inequality holds if and only if $u_{\gamma}=0$, that is equivalent to have $E_{\gamma}=B$. Hence, it is sufficient to choose $\gamma_{0}$ small enough in order to make all the previous arguments work to conclude the proof.
	\end{proof}
	\begin{rem}
		One notices that \autoref{thm:reg-almost-min-frac-per} holds for sequences of sets, while in our proof we have a family of sets indexed by a continuous parameter. This is not an issue: if the final result was not true, then we could find a sequence $\gamma_k\to0$ and a sequence of sets $E_k\in\argmin\FF_{\gamma_k}$ with $E_k\neq B$, but then we could follow the proof of \autoref{mainthm:FF} for this sequence of sets and obtain a contradiction.
	\end{rem}


\section{Attractive term of positive-power type}\label{sec:2}
    This section can be considered as a short appendix to Frank and Lieb's article \cite{franklieb19} where we explain how to modify the arguments present in their paper that are more affected by the choice of a generic function $g$ instead of a negative power. This also shows some common features between the problems \eqref{eq:FF} and \eqref{eq:GG}:
    \begin{itemize}
    	\item within some suitable constraint ranges, one sees that the minimizers are \textit{exactly balls}, and this happens because the attractive term is much stronger than the repulsive one;
    	\item the functional has a good structure that permits to combine a stability inequality for the attractive term with one for the repulsive term in an effective way, providing the expected result.
    \end{itemize}
    Frank and Lieb work with a functional defined on $\mathcal{K}= L^{1}(\R^{N};[0,1])$, and we keep this setting  because the existence of minimizers is quite easy in that class of objects, while it is not clear if we try to minimize $\GG$ in the class of sets. In fact, Frank and Lieb conjecture that in some cases (that are excluded by the hypotheses) the minimizers \textit{cannot} be characteristic functions, as they claim in \cite[Remarks 2]{franklieb19}. 
    Frank and Lieb study the problem \eqref{eq:GG} when $\alpha>0$ and $g(x)=|x|^{-\lambda}$ for $0<\lambda<N-1$, and they prove that the minimizers are balls if the constraint is \textit{big} enough. One could expect a similar behaviour because, if we take $h[t]$ instead of $h$ for some factor $t>0$, then we see that the attractive and the repulsive terms scale respectively by a factor $t^{2N+\alpha}$ and $t^{2N-\lambda}$, and the only minimizers of the attractive term are balls\footnote{This can be seen using Riesz inequality (see \cite[Theorem 3.7]{lieb-loss}).}. As anticipated, our generalization concerns the repulsive part of $\GG$, where we replace the negative power with a function $g:\R^N\setminus\{0\}\to[0,+\infty)$ that satisfies \HI. Notice that the additional integrability condition expressed in \HI\ guarantees that $g\in L^1_{loc}(\R^{N-1})$ since $g$ is assumed to be radial. In fact, we will use this lower dimensional integrability property because many times we will use Fubini and we will need to estimate an integral made on $(N-1)$-dimensional slices. 
    In the end, we are able to use their proofs with minor changes also for this case, as they asserted that could be done. It is worth to point out that our estimates are necessarily less precise since we ignore the exact behaviour of $g$ (being in fact unknown). Nonetheless, they highlight only the necessary features that the model requires in order to retrieve some important features.

    As it is pointed out in \cite{franklieb19}, the key issue is that a regularity theory for the minimizers of $\GG$ is not yet developed, differently from the functional $\FF_{\gamma}$ that enjoys some good property inherited by the $C$-minimizers of the $s$-perimeter. Hence, the main effort is concentrated in proving that the (translated and rescaled) minimizers converge in \textit{Hausdorff} distance to a ball when $m\to\infty$. Once we have this result we can proceed with our program of using the stability results for the attractive term and the repulsive one (respectively \autoref{thm:power-ineq} and \autoref{prop:low-bound-RR}) that allow us to prove that minimizers are exactly balls if $m$ is big enough.\\
    The following inequality for the attractive term plays the role of the quantitative isoperimetric inequality (see \cite[Theorem 5]{franklieb19}):
    \begin{thm}\label{thm:power-ineq}
        Let $\alpha>0$ be fixed. There exists a constant $C=C(N,\alpha)>0$ such that for every function $h\in \mathcal{K}$ with $\norm{h}_{1}=m$ we have that
        \[
            \I_{\alpha}(h) \geq \I_{\alpha}(B[m])+Cm^{2+\alpha/N}A(h)^{2}.
        \]
    \end{thm}
    \begin{rem}\label{rem:asymmetry-to-zero}
        As a consequence, we have that if $m_{k}\to+\infty$ and $h_{k}$ minimizes $\GG$ with constrained volume $m_{k}$, then $A(h_{k})\to 0$. In fact, it is sufficient to apply a simplified version of \autoref{lemma:lipschitz-RR} to $h_1=h$ and $h_2=\Chi{B[m]}$ with $m=\norm{h}_1$. In this case, we define $h^+=h(1-\Chi{B[m]})$ and $h^-=\Chi{B[m]}(1-h)$, so that $h\wedge\Chi{B[m]}=h\Chi{B[m]}$, and we can rewrite the first line in \eqref{eq:RR-difference-identity} as
        \begin{equation}\label{eq:division-function}
            \begin{split}
                \RR(h)-\RR(B[m])&= \RR(h\Chi{B[m]})+\RR(h^{+})+2\RR(h\Chi{B[m]},h^{+})\\
                &\qquad\qquad\qquad\qquad-\RR(h\Chi{B[m]})-\RR(h^{-})-2\RR(h\Chi{B[m]},h^{-})\\
                &=2\RR(\Chi{B[m]},h^{+})-2\RR(\Chi{B[m]},h^{-})+\RR(h^{+})+\RR(h^{-})-2\RR(h^{-},h^{+}).
            \end{split}
        \end{equation}
        We notice that $h^+-h^-=h-\Chi{B[m]}$, and we can rewrite the last three terms as $\RR(h^+,h^+-h^-)+\RR(h^-,h^--h^+)$. Therefore, if $m_k\geq \omega_N$ and if we translate $h_k$ in order to have that $B[m_k]$ is the optimal ball in the definition of $A(h_k)$, then
        \begin{align*}
            |\RR(h_k)-\RR(B[m_k])|&\leq \int \left(2v_{B[m_k]}(x)+v_{h_k^+}(x)+v_{h_k^-}(x)\right)\left|h_k(x)-\Chi{B[m_k]}(x)\right|\dx\\
            &\leq \int Cm_k\left|h_k(x)-\Chi{B[m_k]}(x)\right|\dx\\
            &\leq Cm_k^2A(h_k),
        \end{align*}
        where we used the second part of \autoref{lemma:lipschitz-RR} to pass from the first line to the second. Then it is easy to conclude using the minimality of $h_k$ for $\GG$.
    \end{rem}

    It is useful to have a good diameter bound for the support of any minimizer of $\GG$ with fixed mass, that is contained in the following lemma (whose analogue in the paper is called Lemma 16). To abbreviate the notation, we will always call it \textit{diameter of }$h$ in place of ``diameter of the support of $h$".
    \begin{lemma}\label{lemma:diam-bound-pow}
        Let $\alpha>0$ be fixed and let $g:\R^N\setminus\{0\}\to[0,+\infty)$ be a function satisfying \HH. There exists a constant $C=C(\alpha,g,N)>0$ such that, for any minimizer $h\in\K$ of $\GG$ with constrained volume $\norm{h}_{1}=m\geq \omega_N$, we have that
        \begin{equation}\label{eq:diameter-bound}
            \diam h\leq C m^{1/N}.
        \end{equation}
    \end{lemma}
    \begin{proof}
        We proceed by contradiction. Let us suppose that there exists a sequence of functions $h_{k}\in\K$ such that $d_{k}=\diam h_{k}>k \norm{h_{k}}_{1}^{1/N}$ and $h_{k}$ minimizes $\GG$ with volume constraint $m_{k}=\norm{h_{k}}_{1}\geq \omega_N$. For each $k$ we rotate the function $h_{k}$ in order to have that $d_{k}=\diam(\pi(\supp h_k))$ where $\pi:\R^N\to\R$ is the orthogonal projection on the first axis. We further translate it in order to have
        \[
            \int_{\{\scal{x}{e_1}<0\}}h_{k}(x)\dx=m_{k}/2.
        \]
        If necessary, we substitute $h_k$ with the function $\widetilde{h}_{k}$ defined as $\widetilde{h}_{k}((x_{1},x'))\coloneqq h_{k}((-x_{1},x'))$ for all $(x_{1},x')\in\R\times\R^{N-1}$, in order to have
        \[
            \sup\left\{\scal{x}{e_1} : x\in \supp h_{k}\right\}\geq d_{k}/2.
        \]
        Thus, we can choose a point $t_{k}\in(0,d_k/2)$ such that $0<\int_{\{x_{1}>t_{k}\}}h_{k}(z)\diff z<m_{k}/5$. Then we cut the functions with the hyperplane $\{\scal{x}{e_1}=t_{k}\}$ and compare them with the original ones: if $u_{k}=h_{k}\Chi{\{x_1<t_{k}\}}$, $\eta_{k}=\int_{\{x_{1}>t_{k}\}}h_{k}(x)\dx$ and $\lambda_{k}= (1-\eta_{k}/m_k)^{-1/N}$, then
        \begin{equation}\label{eq:inutile19}
            \begin{split}
                \GG(h_{k})\leq \GG(u_{k}[\lambda_{k}]) &= \mathcal{I}_{\alpha}(u_{k}[\lambda_{k}]) + \RR(u_{k}[\lambda_{k}])\\
                &\leq \lambda_{k}^{2N+\alpha}\mathcal{I}_{\alpha}(u_{k}) + \lambda_{k}^{2N}\RR(u_{k})\\
                &\leq\lambda_{k}^{2N+\alpha}\mathcal{I}_{\alpha}(h_{k}) - \lambda_{k}^{2N+\alpha}\frac{m_{k}}{2}\eta_{k}t_{k}^{\alpha}+\lambda_{k}^{2N}\RR(h_{k}).
            \end{split}
        \end{equation}
        Rearranging that inequality we obtain the following relation:        
        \[
            (\lambda_{k}^{2N+\alpha}-1)\mathcal{I}_{\alpha}(h_{k})+(\lambda_{k}^{2N}-1)\RR(h_{k})\geq \lambda_{k}^{2N+\alpha}\frac{m_{k}\eta_{k}t_{k}^{\alpha}}{2}\geq\frac{m_{k}^{2}t_{k}^{\alpha}}{2}\frac{\eta_{k}}{m_{k}}.
        \]
        Using that $\eta_{k}/m_{k}\leq 1/5$ we can apply the (local) Lipschitz property of the factors in the left hand side to arrive to the following inequality, where we already simplified the positive factor $\eta_{k}/m_{k}$:
        \begin{equation*}
            m_{k}^{2}t_{k}^{\alpha}\leq C(N,\alpha)\mathcal{I}_{\alpha}(h_{k})+C(N)\RR(h_{k}).
        \end{equation*}
        We assumed by contradiction that $d_{k}>km_{k}^{1/N}$, so using the previous inequality with $3t_k>km_k^{1/N}$ and comparing the energy of $h_{k}$ with $\GG(B[m_k])$ we obtain that
        \begin{equation}\label{eq:inutile16}
            k^{\alpha}m_{k}^{2+\alpha/N}\leq C(N,\alpha)\GG(h_{k}) \leq C(N,\alpha)\GG(B[m_{k}]).
        \end{equation}
        We can apply \autoref{lemma:lipschitz-RR} since $m_k\geq \omega_N$ and see that
        \begin{equation}\label{eq:inutile20}
            \GG(B[m_{k}])=\mathcal{I}_{\alpha}(B[m_{k}])+\RR(B[m_{k}]) \leq \left(\diam(B[m_{k}])\right)^{\alpha}m_{k}^{2}+C(g)m_{k}^{2}.
        \end{equation}
        We can plug this inequality into \eqref{eq:inutile16} and obtain that $k^{\alpha}m_{k}^{2+\alpha/N}\leq C(\alpha,g,N)m_{k}^{2+\alpha/N}$, that cannot hold if $k$ is large enough. Therefore we arrived to a contradiction, and then the thesis holds with a constant $C=C(\alpha,g,N)>0$.
    \end{proof}
    \begin{rem}
        Notice that the previous proof provides an explicit upper bound for the diameter that could be expressed in terms of the other constants appearing in the various inequalities.
    \end{rem}
    \begin{rem}\label{rem:diam-bound-GG}
        With the previous computations we can see that, even if we do not have a priori a minimizer, we know that the only possible candidates have the diameter bound resulting from the lemma. In fact, the only energy comparison that we made was $\GG(h)\leq \GG(B[m])$. Thus, it is sufficient to consider a candidate $h$ with that energy inequality and cut it with the proper hyperplanes to find a better competitor that satisfies \eqref{eq:diameter-bound}.
    \end{rem}
    \begin{thm}\label{thm:existence-GG}
        Let $m>0$ and $\alpha>0$ be assigned and let $g:\R^N\setminus\{0\}\to[0,+\infty)$ be a function satisfying \HH. Then there exists a minimizer $h\in \K$ of $\GG$ with constrained volume $\norm{h}_{1}=m$.
    \end{thm}
    \begin{proof}
        Let us take a minimizing sequence $h_{k}\in\K$ for $\GG$ with $\norm{h_{k}}_{1}=m$. Thanks to \autoref{rem:diam-bound-GG} we can suppose without loss of generality that $\supp h_{k}\in B(0,R)$ with $R>0$ being a constant depending on $m$, $N$, $g$ and $\alpha$. Since $|h_{k}|\leq 1$, then up to subsequences $h_k\xrightharpoonup{*}h$ in $L^{\infty}$ for some $h\in L^{\infty}$. Moreover, also $H_k(x,y)=h_k(x)h_k(y)$ weakly-$*$ converge to $H(x,y)=h(x)h(y)$ since we can approximate strongly in $L^1$ any test function $\phi\in L^1(\R^N\times\R^N)$ with functions of the form $\sum_{i\in J}\phi_{1,i}(x)\phi_{2,i}(y)$ with $J$ finite. This is sufficient to conclude using a standard lower semicontinuity argument because the map $(x,y)\mapsto |x-y|^{\alpha}+g(x-y)$ is of class $L^1_{loc}(\R^N\times \R^N)$ and $\supp H_k\subset B(0,R)\times B(0,R)$. Notice that the boundedness of the supports also guarantees that the $L^1$ constraint is satisfied by $h$.   
    \end{proof}
    
    \begin{rem}\label{rem:results-very-singular}
    	One can notice that \autoref{rem:asymmetry-to-zero}, \autoref{lemma:diam-bound-pow} and \autoref{thm:existence-GG} hold without the stronger integrability condition stated in \HI. This ensures that, even if $g(x)=|x|^{-\lambda}$ with $N-1\leq \lambda<N$, we still have that the minimizers exist\footnote{And they could possibly be different from characteristic functions of sets.}, they have bounded diameter and they converge in $L^1$ to a ball as the volume constraint goes to $+\infty$.
    \end{rem}

    The next proposition is the analogue of \cite[Proposition 7]{franklieb19}. Our result is weaker, nevertheless it is sufficient to make the successive arguments work. Notice that our proof is different and more robust since it uses much fewer properties of the repulsive kernel. This result is in the same spirit of \autoref{prop:frac-sobolev-bound-R} since in both cases we compare the energy of competitor that is close to a ball in a suitable sense with the energy of the ball itself.

    \begin{prop}\label{prop:low-bound-RR}
        Let $g:\R^N\setminus\{0\}\to[0,+\infty)$ be a function satisfying \HI. There exists $C(g,N)>0$ such that, for every $\theta\in[0,1/3]$ and for every function $h\in\K$ with
        \[\norm{h}_{1}=m\geq\omega_{N}\qquad \text{and}\qquad\Chi{(1-\theta)B[m]}\leq h\leq\Chi{(1+\theta)B[m]},\]
        we have that $\RR(h)\geq\RR(B[m])-C(g,N)m^{2}\theta^{2}$.
    \end{prop}
    \begin{proof}
        We denote by $R$ the radius of $B[m]$ (notice that $R\geq 1$ since $m\geq \omega_{N}$) and we rewrite the quantity $\RR(h)-\RR(B[m])$ as in \eqref{eq:division-function}. This is useful because $h$ and $\Chi{B[m]}$ are very close, and this produces many cancellations.
        For the sake of brevity we denote by $\mu^+$ and $\mu^-$ respectively the measures $\mu^{+}\coloneqq h^{+}\lebesgue^{N}$ and $\mu^{-}\coloneqq h^{-}\lebesgue^{N}$, where $h^+$ and $h^-$ are the functions appearing in \eqref{eq:division-function}. Then using the function $\psi$ defined in \eqref{eq:psi} we have that
        \begin{equation*}
            \begin{split}
                \RR(h)-\RR(B[m]) &\geq 2\int_{\R^{N}}\psi(R,|x|)\diff\mu^{+}(x)-2\int_{\R^{N}}\psi(R,|x|)\diff \mu^{-}(x)-\RR(h^{+}\!+h^{-})\\
                &=2\int(\psi(R,|x|)-\psi(R,R))\diff\mu^{+}\!-2\int(\psi(R,|x|)-\psi(R,R))\diff\mu^{-}\!-\RR(h^{+}\!+h^{-}),
            \end{split}
        \end{equation*}
        where we used that $\int h^{+}(x)\diff x=\int h^{-}(x)\diff x$. We concentrate ourselves on the last term: if $A=B(0,(1+\theta)R)\setminus B(0,(1-\theta)R)$, then $\RR(h^{+}+h^{-}) \leq \RR(A) = \int_{A}v_{A}(x)\diff x$. By symmetry we can estimate $v_{A}(x)$ only for points of the form $x=se_{1}\in A$. Moreover, we clearly have that $v_{A}(x) = v_{A\setminus Q}(x)+v_{A\cap Q}(x)$, where $Q=Q_{N}(x,1/2)$. We can easily bound the volume of $A$ using that $\theta\in[0,1/3]$:
        \begin{equation}\label{eq:inutile13}
            |A|=|B(0,(1+\theta)R)|-|B(0,(1-\theta)R)| = \omega_{N}\left((1+\theta)^{N}-(1-\theta)^{N}\right)R^{N}\leq C_{N}\theta R^{N},
        \end{equation}
        and then we have that $v_{A\setminus Q}(x)\leq g(1/2)C_{N}\theta R^{N}$. Moreover it is immediate to see that, if we denote any point $x\in\R^{N}$ by $x=(x_{1},x')\in\R\times\R^{N-1}$, then the bound on $\theta$ implies that
        \[
            \haus^{1}(A\cap Q\cap\{x' = y'\})\leq C_{N}\theta R\qquad\forall y'\in\R^{N-1}.
        \]
        Hence, defining $d=C_{N}\theta R$ to be the right hand side of the previous inequality, we can use the monotonicity of $g$ to see that
        \begin{align*}
            v_{A\cap Q}(x) &= \int_{A\cap Q}g(x-y)\dy \leq \int_{Q_{N-1}(0,1/2)}\int_{-d/2}^{d/2}g((s,y'))\diff s\diff y'\\
            &\leq d\int_{Q_{N-1}(0,1/2)}g((0,y'))\diff y' \leq C(g,N)\theta R,
        \end{align*}
        where we used that $g\in L^1_{loc}(\R^{N-1})$ in the last inequality. Combining it with \eqref{eq:inutile13} and using that $R\geq 1$ we have that $\RR(h^{+}+h^{-})\leq \RR(A) \leq C(g,N)\theta^{2}R^{2N}$.\\
        With very similar computations, taking translations of balls instead of dilations, one can see that
        \[
            |\psi(R,|x|)-\psi(R,R)|\leq C(g,N)\theta R(1+R^{N-1})\qquad \forall x\in A,
        \]
        where one needs to use the $(N-1)$-dimensional version of the $L^{\infty}$ bound present in \autoref{lemma:lipschitz-RR} that holds because, as we observed at the beginning of this section, $g\in L^1_{loc}(\R^{N-1})$:
        \begin{equation}\label{eq:strong-uniform-bound-v}
            \int_{E'}g(x)d\haus^{N-1}(x) \leq C(g)\haus^{N-1}(E')\qquad \forall E'\subset\R^{N-1}\times\{0\}\text{ with }\haus^{N-1}(E')\geq \omega_{N-1}.
        \end{equation}
        Putting together the inequalities for all of the terms, using that $R\geq 1$ and exploiting again \eqref{eq:inutile13} together with the fact that $(\supp h^{+}\cup\supp h^{-})\subset A$, we arrive to the conclusion:
        \[
            \RR(h)-\RR(B[m])\geq -C(g,N)\theta^{2}R^{2N}=-C(g,N)\theta^{2}m^{2}.
        \]
    \end{proof}
    The following is a simple technical lemma concerning the function $\Phi_m$ defined in \eqref{eq:potentials}:
    \begin{lemma}\label{lemma:growth}
        Let $\alpha>0$ be fixed and let $g:\R^N\setminus\{0\}\to[0,+\infty)$ be a function satisfying \HI. There exist two constants $m_{0}=m_{0}(\alpha,g,N)>0$ and $C=C(\alpha,g,N)>0$ such that for every $m\geq m_{0}$ we have that
        \begin{gather}\label{eq:weak-monotonicity-Phi}
            \Phi_{m}(r)\leq \Phi_{m}(R) \quad\text{if }r\leq R\qquad \text{and}\qquad\Phi_{m}(r)\geq \Phi_{m}(R)\quad\text{if }r\geq R,\\
            \label{eq:lower-bound-lip-Phi}
            |\Phi_{m}(r)-\Phi_{m}(R)|\geq C R^{N+\alpha-1}\min\{|r-R|,R\} \qquad\forall r\geq0,
        \end{gather}
         where $R$ is the radius of $B[m]$.
    \end{lemma}
    \begin{proof}
        First of all we change variable in the definition of $\Phi_m$:
        \[
            \Phi_{m}(r) = R^{N+\alpha}\int_{B}\left|\frac{r}{R}e_{1}-x\right|^{\alpha}\diff x + R^{N}\int_{B}g\left(re_{1}-Rx\right)\diff x.
        \]
        Now we take $m_{0}\geq \omega_{N}$ (thus $R\geq 1$) and we see that the following inequalities hold true
        \begin{equation}\label{eq:inutile14}
            \begin{split}
                \Phi_{m}(r)-\Phi_{m}(R) &= R^{N+\alpha}(\varphi(r/R)-\varphi(1))\\
                &\qquad+ R^{N}\int_{B}\left[g\left(R\left(\frac{r}{R}e_{1}-x\right)\right)-g(R(e_{1}-x))\right]\diff x\\
                &\geq R^{N+\alpha}(\varphi(r/R)-\varphi(1)) - C(g)R^{N},
            \end{split}
        \end{equation}
        and in the same way
        \[
            \Phi_{m}(R)-\Phi_{m}(r) \geq R^{N+\alpha}(\varphi(1)-\varphi(r/R))-C(g)R^{N}.
        \]
        Moreover, using the change of variables $x=te_{1}-y$, it is easy to see that
        \[
            \varphi'(t) = \alpha\int_{B}(t-\scal{y}{e_{1}})|te_{1}-y|^{\alpha-2}\diff y = \alpha\int_{te_{1}-B}\scal{x}{e_{1}}|x|^{\alpha-2}\diff x,
        \]
        and therefore using the symmetry of $B$ we have that $\varphi$ is of class $C^{1}$, with strictly positive derivative at each point $t>0$. We will denote by $C(\alpha)>0$ a constant such that $\varphi'(t)>C(\alpha)$ for all $t\in[1/3,4/3]$. From the previous observations it follows immediately that both \eqref{eq:weak-monotonicity-Phi} and \eqref{eq:lower-bound-lip-Phi} are valid for $|r-R|\geq R/3$ if we take  $m_{0}$ big enough to have that
        \[
            C(g)\leq \frac{1}{2}R^{\alpha}\min\{\varphi(4/3)-\varphi(1),\varphi(1)-\varphi(1/3)\}.
        \]
        Now we concentrate ourselves on the case $|r-R|\leq R/3$. We treat more carefully the repulsive terms in \eqref{eq:inutile14}, that coincide with
        \begin{equation}\label{eq:inutile15}
            D\coloneqq\int_{re_{1}-B[m]}g(x)\diff x - \int_{Re_{1}-B[m]}g(x)\diff x.
        \end{equation}
        Let us define $\theta=(r-R)/R$, $E=Re_{1}-B[m]$ and let $l=\Span\{e_{1}\}$. Looking at \autoref{figure:2} to better understand the situation, it is immediate to see that
        \[
            \haus^{1}((E\symmdiff (\theta Re_{1}+E))\cap (x+l))\leq 2|\theta| R \qquad \forall x\in\R^{N}.
        \]
        Hence we use that $g\in L^1_{loc}(\R^{N-1})$ to get
        \[
            |D|\leq (2N-2)\omega_{N-2}|\theta| R\int_{-1}^{1}g(se_{1})s^{N-2}\diff s + g(e_{1})|E\symmdiff (\theta Re_{1}+E)|\leq C(g,N)|\theta|(R +R^{N}),
        \]
        \begin{figure}[H]
            \begin{minipage}[c]{0.4\textwidth}
                \includegraphics[width=0.9\textwidth]{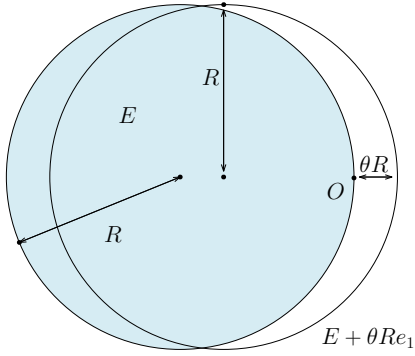}
            \end{minipage}\hfill
            \begin{minipage}[c]{0.5\textwidth}
                \caption{The highlighted region corresponds to the set $E$, while $O$ is the origin of the coordinates used in \eqref{eq:inutile15}.}
                \label{figure:2}
            \end{minipage}
        \end{figure}
        where we used twice the cylindrical coordinates around the $e_1$ axis. Hence we obtain both \eqref{eq:weak-monotonicity-Phi} and \eqref{eq:lower-bound-lip-Phi} if we plug this inequality for $|D|$ into the first line of \eqref{eq:inutile14} and use that $\varphi'(t)\geq C(\alpha)>0$ for $t\in[1/3,4/3]$.
    \end{proof}

    The key estimate, where the integrability property required in \HI\ is fully used, is contained in the following lemma (which is labelled as Lemma 15 in Frank and Lieb's article). With this lemma we gain control on $v_h$ out of a geometric constraint on the support of $h$.
    \begin{lemma}\label{lemma:bound-annulus}
        Let $g:\R^N\setminus\{0\}\to[0,+\infty)$ be a function satisfying \HI. There exists a positive constant $C(g,N)$ such that, for any $R\geq 1$, any $\theta\in[0,1]$ and any function $h\in\K$ that satisfies $\supp h\subset B(0,(1+\theta)R)\setminus B(0,(1-\theta)R)$, we have that
        \begin{equation}\label{eq:inutile}
            \norm{v_h}_{\infty}=\sup_{x\in\R^{N}}\int g(x-y)h(y)\dy \leq C(g,N)\theta R^N.
        \end{equation}
    \end{lemma}
    \begin{rem}
        Our proof goes on quite like that of Frank and Lieb's lemma, but we provide a very rough estimate, where $h$ does not to appear explicitly in the right hand side. Besides this inequality might seem very bad, notice that if we take $h=\Chi{B(0,(1+\theta)R)}-\Chi{B(0,(1-\theta)R)}$ then we see that the bound must be linear in $\theta$ for $\theta\to0$: the left hand side of \eqref{eq:inutile} is larger than $\int g(y)h(y)\dy$, that is larger than $C_Ng((1+\theta)R)\theta R^N$ for some dimensional constant $C_N>0$.
    \end{rem}
    \begin{proof}
        We define the annulus $A\coloneqq B(0,(1+\theta)R)\setminus B(0,(1-\theta)R)$, and since $|A|=\omega_NR^N((1+\theta)^N-(1-\theta)^N)$ we notice that there exists a constant $C_N>0$ such that $C_N^{-1}|A|\leq \theta R^N\leq C_N|A|$. Without loss of generality we can suppose that $|A|\leq \epsilon_N$ for every fixed $\epsilon_N<\omega_N$: if the other case holds, then we denote by $r_N$ the radius of $B[\epsilon_N]$, and we get
        \begin{equation}\label{eq:big-A}
            \begin{split}
                \sup_{x\in\R^{N}}\int g(x-y)h(y)\dy &\leq \int_{B[|A|]}g(y)\dy = \int_{B[\epsilon_N]}g(y)\dy + \int_{B[|A|]\setminus B[\epsilon_N]}g(y)\dy\\
                &\leq C(g) + g(r_Ne_1)|A|\\
                &= \frac{C(g)}{|A|}|A|+g(r_Ne_1)|A|\\
                &\leq \left(\frac{C(g)}{\epsilon_N} + g(r_Ne_1)\right)|A|,
            \end{split}
        \end{equation}
        that is the desired result since $|A|\leq C_N\theta R^N$. The value of $\epsilon_N$ will be fixed later, but it is important to keep in mind that $|A|$ can be taken arbitrarily small. Thus, we need to prove \eqref{eq:inutile} exploiting the particular shape of $A$. In the end, it is sufficient to estimate the following quantity:
        \[
        	S\coloneqq \int_{[-R/2,R/2]^{N-1}}\int_{[-C_N\theta R,C_N\theta R]}g((y',t))\diff y'\diff t,
        \]
        where $C_N>0$ is a geometric constant. In fact, by compactness there exist a constant $K_N>0$ and a family $\{q_1,\ldots,q_{K_N}\}$ of $(N-1)$-dimensional cubes embedded in $\R^N$ such that
        \begin{itemize}
        	\item the center $c_j$ of $q_j$ belongs to $\bdry B(0,R)$ for all $j$;
        	\item their sides have length $R/2$;
        	\item for every $1\leq j\leq K_N$ we have that $q_j\cap \Int(B(0,R)) = \emptyset$;
        	\item if $D_j = \{tc_j+y:t>-1, y\in q_j\}$ and $\pi_j^{\perp}$ is the orthogonal projection onto $\Span\{c_j\}^{\perp}$, then we define the map $\pi_j:D_j\to\R^N$ as
        	\[
                \pi_j(x) = \pi_j^{\perp}(x-c_j)+c_j\sqrt{1-\frac{|\pi^{\perp}(x-c_j)|^2}{R^2}},
            \]
        	so that $\bigcup_{j=1}^{K_N}\pi_j\left(q_j\right) = \bdry B(0,R)$, namely they ``cover'' $\bdry B(0,R)$. Notice that the map $\pi_j$ is just pushing the points of $D_j$ onto $\bdry B(0,R)$ as shown in the left picture in \autoref{figure:covering-sphere}.
        \end{itemize}
        Then, thanks to the positivity of $g$ we can replace the ``curved slabs'' $A\cap D_j$ with some flat slabs $F_j$ (the smallest $N$-dimensional rectangle containing $A\cap D_j$ with sides parallel or orthogonal to $q_j$):
        \begin{equation}\label{eq:inutile30}
        	\begin{split}
        		\int g(x-y)h(y)\dy&\leq \sum_{j=1}^{K_N}\int_{D_j\cap A} g(x-y)h(y)\dy\\
        		&\leq \sum_{j=1}^{K_N}\int_{F_j} g(x-y)\dy\\
        		&\leq K_N \int_{[-R/2,R/2]^{N-1}}\int_{[-C_N\theta R,C_N\theta R]}g((y',t))\diff y'\diff t,
        	\end{split}
        \end{equation}
        where we used the monotonicity of $g$ to pass from the second to the third line. We also highlight that $F_j$ has thickness smaller than $C_N\theta R$ for some constant $C_N$ (see \autoref{figure:covering-sphere}, on the right). In fact, if $\theta\leq 1/10$ this is clearly true, and we know that $\theta\leq C_N|A|/R^N$. Since $R\geq 1$ and $|A|\leq \epsilon_N$, we can choose $\epsilon_N$ so that $\theta\leq 1/10$.
        Then from \eqref{eq:inutile30} it is clear that we need only to control the quantity $S$ defined before.
        \begin{figure}[H]
            \begin{minipage}[c]{0.35\textwidth}
                \includegraphics[width=0.95\textwidth]{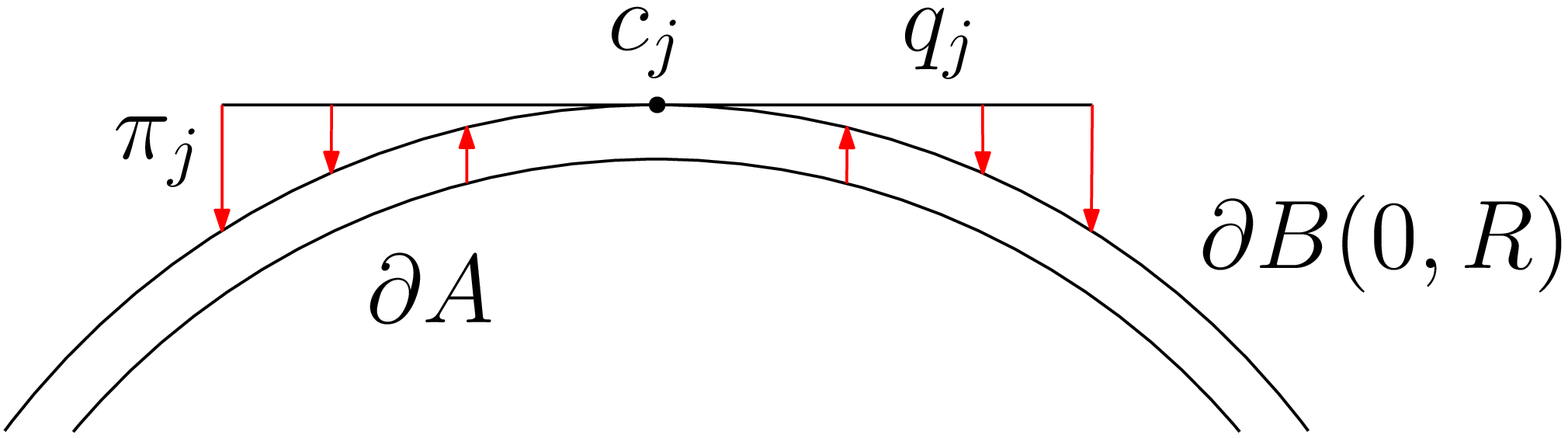}
            \end{minipage}\hfill
            \begin{minipage}[c]{0.35\textwidth}
                \includegraphics[width=0.95\textwidth]{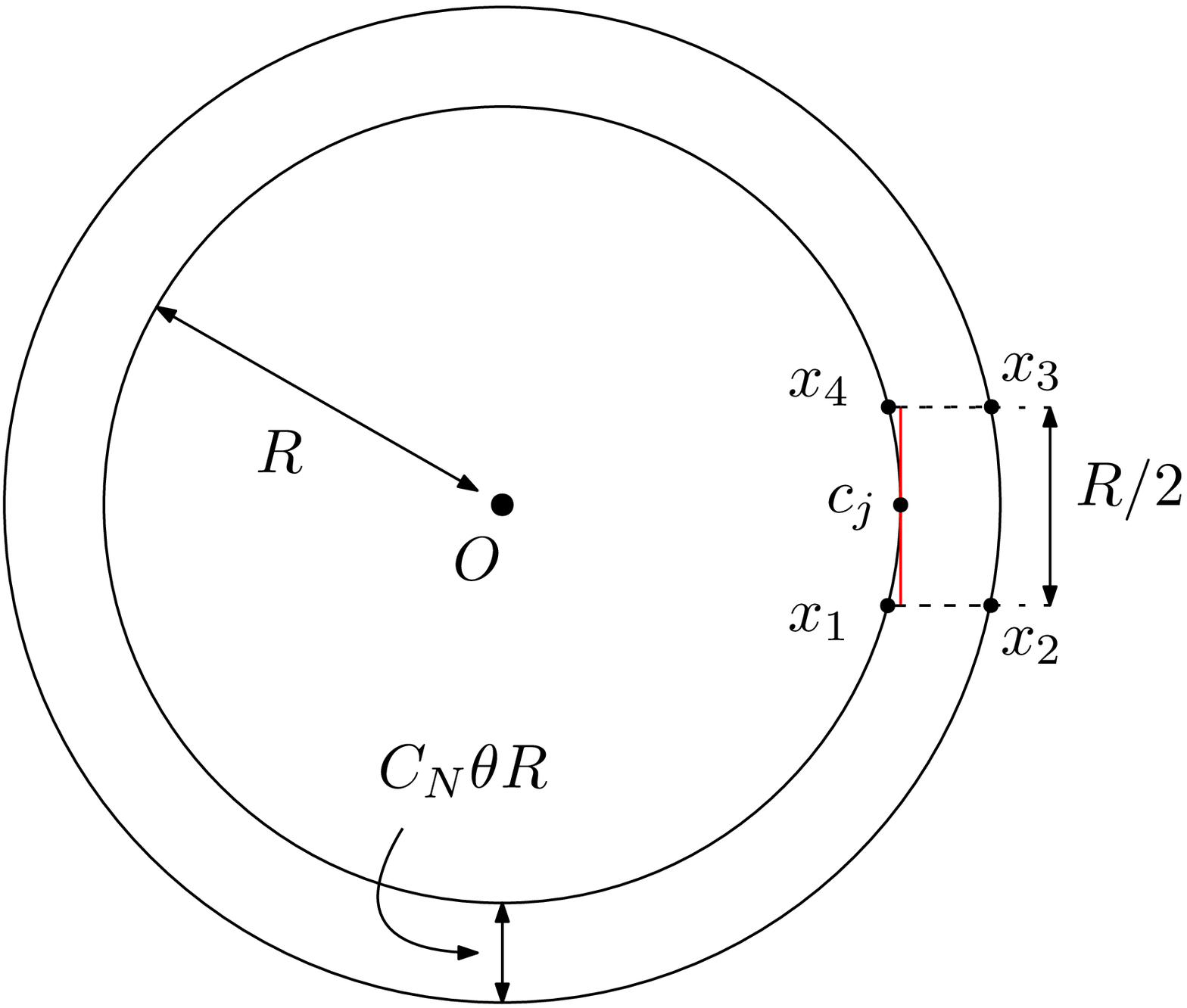}
            \end{minipage}
            \caption{The image on the left represents the map $\pi_j$ with the red arrows, the cube $q_j$ that is the horizontal segment, and one part of the boundary of the annulus $A$. In the figure on the right, the cube $q_j$ is represented by the red vertical segment, while the points $x_1$, $x_2$, $x_3$, $x_4$ denote the corners (since the figure is in 2D) of the outer part of what we call ``curved slab''.}\label{figure:covering-sphere}
        \end{figure}
        Since $g$ is radial and radially decreasing and radial we have that
        \[
        	S \leq \int_{[-R/2,R/2]^{N-1}}\int_{[-C_N\theta R,C_N\theta R]}g((y',0))\diff y'\diff t = 2C_N\theta R\int_{[-R/2,R/2]^{N-1}}g((y',0))\diff y'.
        \]
        Moreover, since $g\in L^1_{loc}(\R^{N-1})$ and $R\geq 1$ we can apply again the strong $L^{\infty}$ bound \eqref{eq:strong-uniform-bound-v} to conclude that
        \[
            S\leq C_N\theta R\cdot C'R^{N-1} = C''\theta R^N,
        \]
        where of course the constant $C''$ depends only on $g$ and the space dimension $N$.
    \end{proof}
    We report for convenience the statement of \cite[Lemma 14]{franklieb19}, that permits to modify a function in order to make it closer to the characteristic function of a ball.
    \begin{lemma}\label{lemma:rounding}
        Let $h\in\K$ be a function with $\norm{h}_{1}=m$, and let $\theta\in[0,1]$. Then there exists $h'\in\K$ with the following properties
        \begin{align}
            &\norm{h'}_{1}=\norm{h}_{1},\label{eq:constr-1}\\
            &\Chi{(1-\theta)B[m]}\leq h'\leq\Chi{(1+\theta)B[m]},\label{eq:constr-2}\\
            & h'(x)\leq h(x)\text{ for }x\not\in B[m],\quad h'(x)\geq h(x)\text{ for }x\in B[m],\label{eq:constr-3}\\
            &\int |h'-\Chi{B[m]}|\dx\leq\int|h-\Chi{B[m]}|\dx,\label{eq:constr-4}\\
            &\int|h-h'|\dx \leq 2\int_{E}|h-h'|\dx\qquad \text{where }E=(1-\theta)B[m]\cup (\R^{N}\setminus (1+\theta)B[m])\label{eq:constr-5}.
        \end{align}
    \end{lemma}
    The following proposition contains the most delicate estimates, and it provides a bound on the Hausdorff distance between $\supp|h-\Chi{B[m]}|$ and $\bdry B[m]$ depending only on $\norm{h-\Chi{B[m]}}_1$ and on $m$. This kind of bound is clearly something special, that holds for minimizers but not for a generic function $h$. The proof that we are going to write is a simple adaptation of the one present in \cite{franklieb19}, where we make some minor changes due to our different estimates. We also point out that Frank and Lieb's proof relies on the fact that $\RR(f)\geq 0$ for any $f\in L^1\cap L^{\infty}$ with bounded support and $\int f\dx=0$. This is true if $g(x)=|x|^{-\lambda}$ for $\lambda\in (0,N)$, but not in general. We do not use this property, and thus we are able to prove the theorem in wider generality.
    \begin{prop}\label{prop:GG-haus}
        Let $\alpha>0$ be fixed and let $g:\R^N\setminus\{0\}\to[0,+\infty)$ be a function satisfying \HI. There exist two constants $C_{0}=C_{0}(\alpha,g,N)>0$ and $m_{0}=m_{0}(\alpha,g,N)>0$ such that, for every minimizer $h\in\K$ of $\GG$ with constrained ``mass" $\norm{h}_{1}=m>m_{0}$, we have that
        \[\Chi{(1-C_{0}A(h))B[m]}\leq h\leq \Chi{(1+C_{0}A(h))B[m]}.\]
    \end{prop}
    \begin{proof}
        The strategy is to build a sequence of competitors out of a given minimizer and estimate precisely the difference in energy between them and the minimizer.\\
        We will determine later the value of the constants $C_{0}$ and $m_{0}$, for now let us fix a minimizer $h_{0}$ with $\norm{h_{0}}_{1}=m>m_{0}$ (that is also the first element of the sequence) and let us properly translate it in order to have that the optimal asymmetry ball is centered in $0$. The following elements of the sequence are defined applying \autoref{lemma:rounding}: if $h_{k}$ has already been built, then $h_{k+1}$ is produced applying that lemma to $h=h_{k}$ and $\theta=2^{-k}$. Now we can study the energy difference between two consecutive functions in the sequence, where the idea is to isolate the terms involving $h_{k+1}-h_{k}$ and to make appear the asymmetry of $h_{k}$ and $h_{k+1}$ in order to use the previous inequalities:
        \begin{align*}
            \GG(h_{k+1})-\GG(h_{k}) &= \GG(h_{k+1}-h_{k},h_{k+1})+\GG(h_{k},h_{k+1})+\GG(h_{k+1}-h_{k},h_{k})-\GG(h_{k+1},h_{k})\\
            &=\GG(h_{k+1}-h_{k},h_{k+1}-\Chi{B[m]}) + \GG(h_{k+1}-h_{k},h_{k}-\Chi{B[m]})\\
            &\qquad+2\GG(h_{k+1}-h_{k},\Chi{B[m]}).
        \end{align*}
        We treat the last term exploiting the comparison estimate of \autoref{lemma:growth} and the properties \eqref{eq:constr-1}, \eqref{eq:constr-3} and \eqref{eq:constr-5}. Let us take $m_{0}$ bigger than the mass constant in that lemma, then denoting by $R$ the radius of $B[m]$ we obtain that
        \begin{align*}
            \GG(h_{k+1}-h_{k},\Chi{B[m]}) &= \int (h_{k+1}-h_{k})\Phi_{m}\dx = \int (h_{k+1}-h_{k})(\Phi_{m}-\Phi_{m}(R))\dx\\
            & =-\int |h_{k+1}-h_{k}||\Phi_{m}-\Phi_{m}(R)|\dx\\
            &\leq-\int_{\{||x|-R|\geq 2^{-k}R\}}|h_{k+1}-h_{k}||\Phi_{m}-\Phi_{m}(R)|\dx\\
            &\leq -C(\alpha,g,N)R^{N+\alpha}2^{-k}\norm{h_{k+1}-h_{k}}_{1}.
        \end{align*}
        Now we concentrate ourselves on the first term of the sum (and the second can be treated using analogous inequalities). In order to write more concise formulas, it is convenient to define the quantity
        \[a_k\coloneqq 2^kR^{-N}\norm{h_k-\Chi{B[m]}}_1\qquad \forall k\in\N,\]
        that up to a multiplicative constant is an approximation of the average of $|h_k-\Chi{B[m]}|$ on the annulus of thickness $2^{-k}$. 
        We separately estimate the attractive and repulsive terms: for the attractive one we can use the diameter bound shown in \autoref{lemma:diam-bound-pow} (that we can apply if $m_{0}\geq \omega_N$) and the condition \eqref{eq:constr-4} to see that
        \begin{align*}
            \mathcal{I}_{\alpha}(h_{k+1}-h_{k},h_{k+1}-\Chi{B[m]}) &\leq C(\alpha,g,N)R^{\alpha}\norm{h_{k+1}-h_{k}}_{1}\cdot2^{-k-1}R^Na_{k+1}\\
            &\leq C(\alpha,g,N)R^{N+\alpha}2^{-k}a_k\norm{h_{k+1}-h_{k}}_{1}.
        \end{align*}
        Of course, the attractive part of the second term can be treated in the same way since we directly obtain that $\mathcal{I}_{\alpha}(h_{k+1}-h_{k},h_{k}-\Chi{B[m]})\leq C(\alpha,g,N)R^{N+\alpha}2^{-k}a_k\norm{h_{k+1}-h_{k}}_{1}$.\\
        For the repulsive term we make use of the property \eqref{eq:constr-2}, that implies that we can apply \autoref{lemma:bound-annulus} to the function $h=|h_{k+1}-\Chi{B[m]}|$ with $\theta=2^{-k}$ (we instead use $\theta=2^{-k+1}$ for the term including $h_{k}-\Chi{B[m]}$). In fact, we arrive to
        \begin{align*}
            \RR(h_{k+1}-h_{k},h_{k+1}-\Chi{B[m]})& \leq \norm{h_{k+1}-h_{k}}_{1}\left[\sup_{x}\int g(x-y)|h_{k+1}(y)-\Chi{B[m]}(y)|\dy\right]\\
            &\leq C(g,N) 2^{-k}R^N\norm{h_{k+1}-h_{k}}_{1}.
        \end{align*}
        Then we can combine the previous inequalities to obtain
        \begin{align}\label{eq:inutile21}
            \GG(h_{k+1})-\GG(h_{k}) \leq - R^{N+\alpha}2^{-k}\norm{h_{k+1}-h_{k}}_{1}\left(C - C' a_{k}-C''R^{-\alpha}\right),
        \end{align}
        where $C$, $C'$ and $C''$ are positive constants that depend only on $\alpha$, $g$ and $N$. Clearly we can take $m_0$ large enough so that $R^{-\alpha}<C/(4C'')$, but thanks to \autoref{rem:asymmetry-to-zero} we can also take $m_{0}$ big enough in order to make $a_{0}<C/(4C')$. Now we are left with two possibilities: either there exists $k_{0}>0$ such that $a_{k_{0}}\geq C/(4C')$, or $a_{k}<C/(4C')$ for every $k\in\N$. If the first eventuality occurs, then we can take the smallest $k_0$ such that $a_{k_{0}}\geq C/(4C')$, and thanks to \eqref{eq:inutile21} we see that
        \begin{equation}\label{eq:energy-comparison-sequence}
        \begin{split}
            \GG(h_{k_{0}})-\GG(h_{0}) &= \sum_{k=0}^{k_{0}-1}\GG(h_{k+1})-\GG(h_{k})\\
            &\leq - R^{N+\alpha}\sum_{k=0}^{k_{0}-1}2^{-k}\norm{h_{k+1}-h_{k}}_{1}\left(C - C' a_{k}-C''R^{-\alpha}\right) \leq0.
        \end{split}
        \end{equation}
        Since $h_{0}$ is a minimizer of $\GG$ with constrained volume $m$, and $\norm{h_{k}}_{1}=m$ for every $k\in\N$, then we necessarily have that $\GG(h_{k_0})-\GG(h_{0})\geq 0$, which is compatible with the previous conditions only if $h_{k}=h_{0}$ for every $k\leq k_{0}$. Therefore, using the property \eqref{eq:constr-2}, we have that
        \[
            \Chi{\left(1-2^{-k_{0}+1}\right)B[m]}\leq h_{0}\leq \Chi{\left(1+2^{-k_{0}+1}\right)B[m]},
        \]
        and so we need only to estimate $2^{-k_{0}}$ in order to prove the result. In this case we have that $a_{k_{0}}\geq C/(4C')$, therefore using \eqref{eq:constr-4} we arrive to
        \[
            2^{-k_{0}}\leq \frac{4C'}{C}R^{-N}\norm{h_{k_{0}}-\Chi{B[m]}}_{1}\leq \frac{4C'}{C}R^{-N}\norm{h_{0}-\Chi{B[m]}}_{1}\leq\frac{4C'\omega_{N}}{C}A(h_{0})
        \]
        that is the desired result.\\
        If instead $a_k<C/(4C')$ for every $k\in\N$, then we apply \eqref{eq:energy-comparison-sequence} to a generic index $k_0\in\N$. If $h_{k+1}\neq h_k$ for some $k<k_0$ then the last inequality is strict, but this is impossible since $h_0$ minimizes $\GG$. Therefore $h_k=h_0$ for every $k\in\N$, and by construction
        \[
            \Chi{\left(1-2^{-k+1}\right)B[m]}\leq h_{k}\leq \Chi{\left(1+2^{-k+1}\right)B[m]}\qquad \forall k\in\N.
        \]
        As a consequence $h_{0}=\Chi{B[m]}$, that clearly satisfies the inequality in the statement. In the end, we can choose $m_{0}$ big enough in order to make the previous arguments work and the constant in the statement is $C_{0}=8C'\omega_{N}/C$.
    \end{proof}

    \begin{proof}[Proof of \autoref{mainthm:GG}]
        We take $m_{0}$ to be the maximum mass threshold appearing in \autoref{lemma:diam-bound-pow}, \autoref{prop:low-bound-RR}, \autoref{lemma:growth} and \autoref{prop:GG-haus}. Given $m>m_{0}$ we take any minimizer $h$ of $\GG$ with $\norm{h}_{1}=m$ and optimal asymmetry ball centered in the origin. Using \autoref{thm:power-ineq} we obtain that
        \[
            Cm^{2+\alpha/N}A(h)^{2}\leq \mathcal{I}_{\alpha}(h)-\mathcal{I}_{\alpha}(B[m])\leq \RR(B[m])-\RR(h)
        \]
        for some constant $C=C(N,g,\alpha)$. Thanks to \autoref{rem:asymmetry-to-zero} we can take $m_0$ big enough in order to have $A(h)$ as small as we want. Then combining \autoref{prop:GG-haus} and \autoref{prop:low-bound-RR} (that we can apply because of the small asymmetry) we have that $\RR(B[m])-\RR(h)\leq C' m^{2}A(h)^{2}$, and therefore we have that $Cm^{\alpha/N}A(h)^{2}\leq C' A(h)^{2}$. Enlarging $m_{0}$ if necessary, we see that the last inequality can hold only if $A(h)=0$, that is precisely the thesis. 
    \end{proof}
\bibliographystyle{abbrv}
\bibliography{bibliografy.bib}
\end{document}